\documentclass[leqno]{article}
\usepackage[frenchb,english]{babel}
\usepackage[T1]{fontenc}

\usepackage{amsmath}
\usepackage{amssymb}
\usepackage{amsfonts}
\usepackage{enumerate}
\usepackage{vmargin}
\usepackage[all]{xy}
\usepackage{mathrsfs}
\usepackage{mathtools}
\usepackage{lmodern}
\usepackage[colorlinks=true,pagebackref=true]{hyperref}
\usepackage{comment}
\setmarginsrb{3cm}{3cm}{3.5cm}{3cm}{0cm}{0cm}{1.5cm}{3cm}

\newcommand{\B}{\ensuremath{\mathrm{B}}}
\newcommand{\C}{\ensuremath{\mathbb{C}}}
\newcommand{\E}{\ensuremath{\mathbb{E}}}
\newcommand{\D}{\ensuremath{\mathbb{D}}}

\newcommand{\K}{\mathrm{K}} 
\let\L\relax 
\newcommand{\L}{\mathrm{L}}
\newcommand{\M}{\mathrm{M}}
\newcommand{\N}{\ensuremath{\mathcal{N}}}

\newcommand{\R}{\ensuremath{\mathbb{R}}}
\newcommand{\X}{\mathrm{X}}
\let\S\relax 
\newcommand{\S}{\mathrm{S}}
\let\P\relax 
\newcommand{\P}{\mathrm{P}}

\newcommand{\T}{\ensuremath{\mathbb{T}}} 

\newcommand{\Z}{\ensuremath{\mathbb{Z}}}

\newcommand{\w}{\mathrm{w}} 
\renewcommand{\d}{\mathop{}\mathopen{}\mathrm{d}} 
\newcommand{\loc}{\mathrm{loc}} 

\newcommand{\Id}{\mathrm{Id}} 
\newcommand{\VN}{\mathrm{VN}} 
\newcommand{\cb}{\mathrm{cb}}

\newcommand{\OK}{\mathrm{OK}}
 \newcommand{\HS}{\mathrm{HS}} 
\newcommand{\ALSS}{\mathrm{ALSS}}
\newcommand{\Sp}{\mathrm{Sp}}

\renewcommand{\leq}{\ensuremath{\leqslant}}
\renewcommand{\geq}{\ensuremath{\geqslant}}
\newcommand{\qed}{\hfill \vrule height6pt  width6pt depth0pt}
\newcommand{\norm}[1]{\left\Vert#1\right\Vert}
\newcommand{\bnorm}[1]{ \big\| #1  \big\|}
\newcommand{\Bnorm}[1]{ \Big\| #1  \Big\|}
\newcommand{\bgnorm}[1]{ \bigg\| #1  \bigg\|}
\newcommand{\Bgnorm}[1]{ \Bigg\| #1  \Bigg\|}
\newcommand{\xra}{\xrightarrow} 
\newcommand{\co}{\colon}

\newcommand{\ot}{\otimes}
\newcommand{\ovl}{\overline}

\newcommand{\ul}{\mathcal{U}}

\newcommand{\ov}{\overset} 
\newcommand{\epsi}{\varepsilon}
\newcommand{\QWEP}{\mathrm{QWEP}}

\DeclareMathOperator{\Ran}{Ran} 
\DeclareMathOperator{\tr}{Tr} 

\newtheorem{thm}{Theorem}
\newtheorem{defi}[thm]{Definition}
\newtheorem{prop}[thm]{Proposition}
\newtheorem{conj}[thm]{Conjecture}

\newtheorem{cor}[thm]{Corollary}
\newtheorem{lemma}[thm]{Lemma}

\newtheorem{remark}[thm]{Remark}

\newenvironment{proof}[1][]{\noindent {\it Proof #1} : }{\hbox{~}\qed
\smallskip
}

\begin{document}
\selectlanguage{english}
\title{\bfseries{On analyticity of semigroups on Bochner spaces and on vector-valued noncommutative $\L^p$-spaces}}
\date{}
\author{\bfseries{C\'edric Arhancet}}


\maketitle

\begin{abstract}
We show that the analyticity of semigroups $(T_t)_{t \geq 0}$ of (not necessarily positive) selfadjoint contractive Fourier multipliers on $\L^p$-spaces of any  abelian locally compact group is preserved by the tensorisation of the identity operator $\Id_X$ of a Banach space $X$ for a large class of K-convex Banach spaces, answering partially a conjecture of Pisier. The result is even new for semigroups of Fourier multipliers acting on $\L^p(\R^n)$. The proof relies on the use of noncommutative Banach spaces and we give a more general result for semigroups of Fourier multipliers acting on noncommutative $\L^p$-spaces. Finally, we also give a somewhat different version of this result in the discrete case, i.e. for Ritt operators. 

\bigskip
\end{abstract}


\makeatletter
 \renewcommand{\@makefntext}[1]{#1}
 \makeatother


\section{Introduction}

In the early eighties, in a famous paper on the geometry of Banach spaces, Pisier \cite[Theorem~2.1]{Pis3} showed that a Banach space $X$ does not contain $\ell^1_n$'s uniformly if and only if the tensorisation $P \ot \Id_X$ of the Rademacher projection
\begin{equation}
\label{Projection-Rademacher}
\begin{array}{cccc}
   P   \co &  \L^2(\Omega_0)   &  \longrightarrow   &  \L^2(\Omega_0)  \\
           &   f  &  \longmapsto       & \displaystyle \sum_{k=1}^\infty \Big(\int_\Omega f\varepsilon_k \Big)\varepsilon_k \\
\end{array}
\end{equation}
induces a bounded operator on the Bochner space $\L^2(\Omega_0,X)$ where $\Omega_0$ is a probability space and where $\epsi_1,\epsi_2,\ldots$ is a sequence of independant random variables with $\P(\epsi_k=1)=\P(\epsi_k=-1)=\frac{1}{2}$. Such a Banach space $X$ is called K-convex. The heart of his proof relies on the fact, proved by himself in his article, that if $X$ is a K-convex Banach space then any weak* continuous semigroup $(T_t)_{t \geq 0}$ of \textit{positive unital\footnote{\thefootnote.That means that $T_t(1)=1$ for any $t \geq 0$.}} selfadjoint Fourier multipliers on a locally compact abelian group $G$ induces a strongly continuous bounded analytic semigroup $(T_t \ot \Id_X)_{t \geq 0}$ of contractions on the Bochner space $\L^p(G,X)$ for any $1 < p < \infty$. In 1981, in the seminars {\cite{Pis1} and \cite{Pis2} which announced the results of his paper, he stated several natural questions raised by his work. In particular, he conjectured \cite[page 17]{Pis1} that the same property holds for any weak* continuous semigroup $(T_t)_{t \geq 0}$ of selfadjoint \textit{contractive} operators on $\L^\infty(\Omega)$ where $\Omega$ is a measure space. Note that it is well-known \cite[III2 Theorem 1]{Ste} that such a semigroup induces a strongly continuous bounded analytic semigroup of contractions on the associated $\L^p$-space $\L^p(\Omega)$ and the conjecture says that the property of analyticity is preserved by the tensorisation of the identity $\Id_X$ of a K-convex Banach space $X$. Finally, Pisier showed that the $\K$-convexity is necessary for some semigroups of Fourier multipliers. 

We refer to the authoritative book \cite[Problem P.4]{HvNVW2} where this problem is explicitly stated with the additional assumptions of positivity and preservation of the unit, to \cite[Problem 11]{Xu} for a more general question, to \cite{Arh2}, \cite{Arh3}, \cite[Chapter 13]{DJT}, \cite{Fac}, \cite{Hin}, \cite{LLM}, \cite{Mau}, \cite{Pis1}, \cite{Pis2} and to the recent preprint \cite{Xu2}.

Using noncommutative Banach space theory (see \cite{ER}, \cite{Pau} and \cite{Pis7}), a quantised theory of Banach spaces, we are able to make progress on this purely Banach space question. First of all, let us recall that a noncommutative Banach space $E$ is just a Banach space equipped with an isometric linear embedding $E \hookrightarrow \B(H)$ into the space of bounded operators on a Hilbert space $H$, that is a realization of $E$ as a space of bounded operators. For this reason, noncommutative Banach spaces are called operator spaces. The morphisms in this category are the\textit{ completely bounded} maps. Moreover, the notion of $\K$-convexity admits a noncommutative analog for operator spaces, the $\OK$-convexity. An operator space $E$ is $\OK$-convex if the Rademacher projection \eqref{Projection-Rademacher} induces a \textit{completely bounded} map $P \co \L^2(\Omega_0,E) \to \L^2(\Omega_0,E)$. This notion was introduced by \cite{JP}. The following result describes our first main result (Theorem \ref{Th-amenable-semigroup}) in the very particular case of the group $\R^n$
.

\begin{thm}
\label{Th-main-Rn}
Let $(T_t)_{t \geq 0}$ be a weak* continuous semigroup of (not necessarily positive) selfadjoint contractive Fourier multipliers on $\L^\infty(\R^n)$. Suppose that $X$ is a $\K$-convex Banach space isomorphic to a $\OK$-convex operator space. Consider $1<p<\infty$. Then $(T_t)_{t \geq 0}$ induces a strongly continuous bounded analytic semigroup $(T_t \ot \Id_X)_{t \geq 0}$ of contractions on the Bochner space $\L^p(\R^n,X)$.
\end{thm}

Note that it is unknown if any $\K$-convex Banach space $X$ is isomorphic to a $\OK$-convex operator space and we conjecture that is this is indeed the case, see Conjecture \ref{Conj}. From our point of view, our result reduce the analysis of the analyticity of semigroups of selfadjoint contractive Fourier multipliers on Bochner spaces to a independent intriguing question concerning the links between the category of Banach spaces and the category of noncommutative Banach spaces (i.e. operator spaces). Note, in most concrete cases, it is easy to check the additional condition on the $\K$-convex $X$ of the theorem since in Section \ref{K'-convex} we show very good stability properties for this assumption. 

Our result admits a generalization for (second countable) abelian locally compact groups, see Corollary \ref{Cor-abelien-semigroups} and even more generally for (unimodular second countable) amenable locally compact groups and their noncommutative $\L^p$-spaces, see Theorem \ref{Th-amenable-semigroup}. In this last context, the assumption reduces purely and simply to the $\OK$-convexity. Consequently, we have a better understanding of the situation in the noncommutative case than in the classical case of $\L^p$-spaces of measures spaces and it is rather surprising. Note that in \cite{Arh3}, we have proved a particular case in the case of amenable discrete groups and that in \cite{Arh5}, we will examine the case of non-amenable groups. Finally, at our knowledge, our paper and our previous work \cite{Arh3} are the only papers which give applications of noncommutative Banach space geometry to classical Banach space theory which are not covered by classical methods.

We will also establish an (somewhat different) analogue of our main result for some Fourier multipliers which are Ritt operators acting on $\L^p$-spaces, see Theorem \ref{Main-result} and a similar result for Schur multipliers, see Theorem \ref{Th-main-Schur}. The class of Ritt operators can be regarded as the discrete analogue of the class of bounded analytic semigroups. Again, our noncommutative methods have consequences in classical setting of $\L^p$-spaces of measure spaces and allows us to go beyond the main result of \cite{LLM}. We refer to \cite{Arh4}, \cite{AFM}, \cite{ALM}, \cite{Bl1}, \cite{Bl2}, \cite[pages 471-472]{HvNVW2}, \cite{LLM}, \cite{LM1}, \cite{Lyu}, \cite{NaZ} and to the recent preprint \cite{Xu2} and references therein for more information on Ritt operators. 

The paper is organized as follows. Section \ref{sec:preliminaries} gives background and describes indispensable tools. We introduce here some notions which are relevant to our paper. 


The next section \ref{sec:main} contains a proof of our main result Theorem \ref{Th-amenable-semigroup} for semigroups and we also describe implications for semigroups of Fourier multipliers on \textit{abelian} locally compact groups. Finally, in Section \ref{Sec-Ritt}, we describe and show results for the discrete case, i.e. for some Fourier multipliers or Schur multipliers which are Ritt operators. In Section \ref{K'-convex}, we examine the class of $\K$-convex Banach spaces which are isomorphic to $\OK$-convex operator spaces. We show the stability under the usual operations: duality, interpolation, ultrapower, etc.

\section{Peliminaries}
\label{sec:preliminaries}

\paragraph{Bounded analytic semigroups} Let $X$ be a Banach space. A strongly continuous semigroup $(T_t)_{t\geq 0}$ is called bounded analytic \cite[Definition 3.7.3]{ABHN} if there exist $0<\theta<\frac{\pi}{2}$ and a bounded holomorphic extension
$$
\begin{array}{cccc}
    &  \Sigma_\theta   &  \longrightarrow   &  \B(X)  \\
    &  z   &  \longmapsto       &  T_z  \\
\end{array}
$$
where $\Sigma_\theta=\{z\in\C^*\ :\ \vert{\arg}(z)\vert <\theta\}$ denotes the open sector of angle $2\theta$ around the positive real axis $\R_+$. We refer to the books \cite{ABHN}, \cite{HvNVW2}, \cite{EN} and \cite{Haa} for more information. If $A$ is an unbounded operator on $X$ then $-A$ is the infinitesimal generator of a bounded analytic semigroup if and only if $A$ is sectorial of type $<\frac{\pi}{2}$. We warn the reader that the terminology varies somewhat in the litterature and that the notions of ``bounded analytic semigroup'' are ``analytic semigroup'' are different.

We need the following theorem which is a corollary \cite[Theorem 1.3 and Footnote (1)]{Pis3} of a result of Beurling \cite[Theorem III]{Beu}, see also \cite[Theorem 2.1]{Pis1}, \cite[Corollary 2.5]{Fac} and \cite{Hin}. This result gives a sufficient condition on a semigroup of contractions  to ensure that this semigroup is bounded analytic.

\begin{thm}
\label{Th-de-Beurling} Let $X$ be a Banach space. Let $(T_t)_{t\geq 0}$ be a strongly continuous semigroup of contractions on $X$. Suppose that there exists some integer $n \geq 1$ such that for any $t> 0$
\begin{equation*}\label{}
\bnorm{(\Id_{X}-T_t)^n}_{X\to X} < 2^{n}.
\end{equation*}
Then the semigroup $(T_t)_{t \geq 0}$ is bounded analytic.
\end{thm}

\paragraph{Operator spaces and noncommutative $\L^p$-spaces} The readers are referred to \cite{ER}, \cite{Pau}, \cite{Pis5} and \cite{Pis7} for details on operator spaces and completely bounded maps and to the survey \cite{PiX} for noncommutative $\L^p$-spaces and references therein.

If $T \co E \to F$ is a completely bounded map between two operators spaces $E$ and $F$, we denote by $\norm{T}_{\cb, E \to F}$ its completely bounded norm.

Note the following classical extension properties of some linear maps between noncommutative $\L^p$-spaces. See e.g. \cite[Lemma 3.20]{ArK} for the second part.

\begin{prop}
\label{prop-tensorisation of CP maps}
Let $M$ and $N$ be von Neumann algebras equipped with normal semifinite faithful traces.
\begin{enumerate}
	\item Let $T \co M \to N$ be a trace preserving unital normal completely positive map. Suppose $1 \leq p < \infty$. Then $T$ induces a complete contraction $T \co \L^p(M) \to \L^p(N)$.
	\item Suppose that $M$ and $N$ are approximately finite-dimensional. Let $E$ be an operator space. Let $T \co M \to N$ be a complete contraction that also induces a complete contraction on $\L^1(M)$. Suppose $1 \leq p \leq \infty$. Then the operator $T \ot \Id_E$ induces a completely contractive operator from $\L^p(M,E)$ into $\L^p(N,E)$.
\end{enumerate}
\end{prop}
Recall the notation $T^\circ(x)=T(x^*)^*$ for a map $T \co \L^p(M) \to \L^p(N)$. 

Let $M$ be a von Neumann algebra equipped with a normal semifinite faithful trace $\tau$. Suppose that $T \co M \to M$ is a normal contraction. We say that $T$ is selfadjoint if for any $x,y \in M \cap \L^1(M)$ we have
$$
\tau\big(T(x)y^*\big)=\tau\big(x(T(y))^*\big).
$$
In this case, it is not hard to show that the restriction $T|M \cap
\L^1(M)$ extends to a contraction $T \co \L^1(M) \to \L^1(M)$. By complex interpolation, for any $ 1\leq p <\infty$, we obtain a contractive map $T \co \L^p(M) \to \L^p(M)$. Moreover, the operator $T \co \L^2(M) \to \L^2(M)$ is selfadjoint. If $T \co M \to M$ is actually a normal selfadjoint complete contraction, it is easy to see that the map $T \co \L^p(M) \to \L^p(M)$ is completely contractive for any $1 \leq p < \infty$. 

Recall that a map $T \co \L^p(M) \to \L^p(N)$ between noncommutative $\L^p$-spaces, associated with approximately finite-dimensional von Neumann algebras $M$ and $N$, is regular if for any operator space $E$, the map $T \ot \Id_E$ induces a bounded operator between the vector-valued noncommutative $\L^p$-spaces $\L^p(M,E)$ and $\L^p(N,E)$. 

\paragraph{Groups von Neumann algebras and Fourier multipliers}
Here we use the notation of \cite{ArK} and we refer to this paper for more information and references. Let $G$ be a locally compact group equipped with a fixed left invariant Haar measure $\mu_G$. For a complex measurable function $g \co G \to \C$, we write $\lambda(g)$ for the left convolution operator (in general unbounded) by $g$ on $\L^2(G)$. This means that the domain of $\lambda(g)$ consists of all $f$ of $\L^2(G)$ for which the integral $(g*f)(t)=\int_G g(s)f(s^{-1}t) \d\mu_G(s)$ exists for almost all $t \in G$ and for which the resulting function $g*f$ belongs to $\L^2(G)$, and for such $f$, we let $\lambda(g)f=g*f$. Finally, by \cite[Corollary 20.14]{HR}, each $g \in \L^1(G)$ induces a bounded operator $\lambda(g) \co \L^2(G) \to \L^2(G)$. 

Let $\VN(G)$ be the von Neumann algebra generated by the set $\big\{\lambda(g) : g \in \L^1(G)\big\}$. It is called the group von Neumann algebra of $G$ and is equal to the von Neumann algebra generated by the set $\{\lambda_s : s \in G\}$ where 
\begin{equation*}
\label{Left-translation}
   \lambda_s  \co  \begin{cases}
  \L^2(G)   &  \longrightarrow    \L^2(G)  \\
            f   &  \longmapsto        (t \mapsto f(s^{-1}t))
\end{cases}
\end{equation*}
is the left translation by $s$. 
Recall that for any $g \in \L^1(G)$ we have $\lambda(g)=\int_{G} g(s)\lambda_s \d\mu_G(s)$ where the latter integral is understood in the weak operator sense. 

In this paper, we shall limit ourselves to unimodular groups to avoid technical issues concerning modular theory. Let $G$ be a unimodular locally compact group. Recall that in this case the Plancherel weight $\tau_G$ on $\VN(G)$ is tracial. 
 Suppose $1 \leq p \leq \infty$.  We say that a (weak* continuous if $p=\infty$) bounded operator $T \co \L^p(\VN(G)) \to \L^p(\VN(G))$ is a $\L^p$-Fourier multiplier \cite[Definition 6.3]{ArK} if there exists a locally 2-integrable function $\phi \in \L^2_{\loc}(G)$ such that for any $f \in C_c(G) * C_c(G)$ ($f \in C_c(G)$ if $p=\infty$) the element $\int_{G} \phi(s) f(s) \lambda_s \d \mu_G(s)$ belongs to $\L^p(\VN(G))$ and 
\begin{equation*}
T\bigg(\int_{G} f(s) \lambda_s \d \mu_G(s)\bigg) 
= \int_{G} \phi(s) f(s) \lambda_s \d \mu_G(s), \quad \text{i.e.} \quad T(\lambda(f)) =\lambda(\phi f).
\end{equation*}
In this case, we let $M_\phi\ov{\textrm{def}}=T$.

\paragraph{Schur multipliers} Suppose $1 \leq p <\infty$. Let us remind the definition of a Schur multiplier on $S^p_\Omega=\L^p(\B(\L^2(\Omega)))$ where $(\Omega,\mu)$ is a (localizable) measure space \cite[Section 1.2]{LaS}. If $f \in \L^2(\Omega \times \Omega)$, we denote by $K_f \co \L^2(\Omega) \to \L^2(\Omega)$, $u \mapsto \int_\Omega f(z,\cdot) u(z) \d z$ the integral operator with kernel $f$. We say that a measurable function $\phi \co \Omega \times \Omega \to \C$ induces a bounded Schur multiplier on $S^p_\Omega$ if for any $f \in \L^2(\Omega \times \Omega)$ satisfying $K_f \in S^p_\Omega$ we have $K_{\phi f} \in S^p_\Omega$ and if the map $S^2_\Omega \cap S^p_\Omega \to S^p_\Omega$, $K_f \mapsto K_{\phi f}$ extends to a bounded map $M_\phi$ from $S^p_\Omega$ into $S^p_\Omega$ called the Schur multiplier associated with $\phi$. We refer to the surveys \cite{ToT1} and \cite{Tod1} for the case $p=\infty$. 

Let $G$ be a locally compact group. The right regular representation $\rho \co G \to \B(\L^2(G))$ is given by $(\rho_t\xi)(s) = \xi(st)$. Recall that $\rho$ is a strongly continuous unitary representation. Using the notation $\mathrm{Ad}_{\rho_s}^p \co S^p_G \to S^p_G$, $x \mapsto \rho_s x \rho_{s^{-1}}$,  we say that a bounded Schur multiplier $M_\phi \co S^p_G \to S^p_G$ is a Herz-Schur multiplier if for any $s \in G$ we have $M_\phi\mathrm{Ad}_{\rho_s}^p =\mathrm{Ad}_{\rho_s}^p M_\phi$. In this case, there exists a measurable function $\varphi \co G \to \C$ such that $\phi(r,s)=\varphi(rs^{-1})$ for almost every $r,s \in G$ and we let $M_{\varphi}^{\HS}\ov{\textrm{def}}=M_{\phi}$. For any integer $n$, with obvious notations, we have of course
\begin{equation}
\label{equavarphi}
  \big(M_{\varphi}^{\HS}\big)^n
	=M_{\varphi^n}^{\HS}
	\qquad \text{and} \qquad
	\beta M_{\varphi}^{\HS}+M_{\psi}^{\HS}
	=M_{\beta\varphi+\psi}^{\HS}, \qquad \beta \in \C.
\end{equation}



\paragraph{Transference} We need the following transfer results \cite{NR} \cite{CDS} between Fourier multipliers and Schur multipliers. Let $E$ be an operator space. For any $1 \leq p \leq \infty$, if $G$ is a (unimodular) second countable \footnote{\thefootnote. We warn the reader that the proof of \cite[Theorem 5.2]{CDS} is only valid for second countable groups. The proof uses Lebesgue's dominated convergence theorem in the last line of page 7007 and this result does not admit a generalization for nets.}  amenable locally compact group, by a straightforward generalization of \cite{CDS}, we have the following relations between Fourier multipliers and Schur multipliers\
\begin{equation}
\label{Equa-transfer-leq}
\bnorm{M_\varphi \ot \Id_E}_{\L^p(\VN(G),E) \to \L^p(\VN(G),E)}
\leq \bnorm{M^{\HS}_{\varphi}\ot \Id_E}_{S^p_G(E) \to S^p_G(E)}
\end{equation}
and
\begin{equation}
\label{Equa-transfer=}
\bnorm{M_\varphi \ot \Id_E}_{\cb,\L^p(\VN(G),E) \to \L^p(\VN(G),E)}
=\bnorm{M^{\HS}_{\varphi} \ot \Id_E}_{\cb,S^p_G(E) \to S^p_G(E)}.
\end{equation}

The following is a vector-valued version of \cite[Theorem 1.19]{LaS} obtained with a similar argument.

\begin{prop}
\label{Transference-Schur-discrete-continuous}
Let $\mu$ be a $\sigma$-finite Radon measure on a locally compact space $\Omega$ and let $\varphi \co \Omega \times \Omega \to \C$ be a continuous function. Suppose that $E$ is an operator space. Let $1 \leq p \leq \infty$ and $C \geq 0$. The following are equivalent:
\begin{enumerate}
	\item $\varphi$ induces a bounded Schur multiplier $M_\varphi \co S^p_\Omega(E) \to S^p_\Omega(E)$ with norm less than $C$.
	
	\item For any finite family $\alpha=\{x_1,\ldots,x_{m_\alpha}\}$ of distinct elements of $\Omega$ belonging to the support of $\mu$, the Schur multiplier $M_{[\varphi(x_i,x_j)]_{1 \leq i,j \leq m_\alpha}}$ is bounded on $S^p_{m_\alpha}(E)$ with norm less than $C$.
\end{enumerate}
\end{prop}

\paragraph{Comparison between norms of powers and tensor powers of discrete Schur multipliers} We will use the useful following fundamental observation \cite[Lemma 3.3]{Arh3}. 

\begin{lemma}
\label{Prop Fell absorption} 
Suppose $1 \leq p \leq \infty$. Let $E$ be an operator space. Let $I$ be an index set equipped with the counting measure. For any regular Schur multiplier, $M_\phi \co S^p_I \to S^p_I$ and any positive integer $n \geq 1$ we have
\begin{equation}
\label{Majoration-norme-multiplicateur-1}
\bnorm{(M_\phi)^n \ot \Id_E}_{S^p_I(E) \to S^p_I(E)}
\leq \bnorm{(M_\phi)^{\ot n} \ot \Id_E}_{S^p(\B(\ell^2_I)^{\ot n},E) \to S^p(\B(\ell^2_I)^{\ot n},E)}.
\end{equation}
\end{lemma}

\paragraph{$\K$-convexity and $\OK$-convexity} If $X$ is a $\K$-convex Banach space, we introduce the constant of $\K$-convexity of $X$ by
$$
\K(X)
\ov{\textrm{def}}=\norm{P_2 \ot \Id_X}_{\L^2(\Omega_0,X) \to \L^2(\Omega_0,X)}.
$$
It is well-known, e.g. \cite[Theorem 1.12]{DJT} that the map \eqref{Projection-Rademacher} extends to a bounded projection $P_p \co \L^p(\Omega_0) \to \L^p(\Omega_0)$ and that a Banach space $X$ is $\K$-convex if and only if $P_p \ot \Id_X$ extends to a bounded operator on $\L^p(\Omega_0,X)$. Moreover, in this case, we have for some universal constants $0<A_p<B_p$
\begin{equation}
\label{K-convex-Lp}
A_p\, \K(X)
\leq \bnorm{P_p \ot \Id_X}_{\L^p(\Omega_0,X) \to \L^p(\Omega_0,X)} 
\leq B_p\, \K(X).	
\end{equation}
We say that an operator space $E$ is $\OK$-convex if the vector-valued Schatten space $S^p(E)$ is $\K$-convex for some (equivalently all) $1<p<\infty$.

\begin{lemma}
\label{Lemma-K-convex}
Let $1<p<\infty$ and let $E$ be an $\OK$-convex operator space. There exists a constant $C>0$ such that whenever $M$ is an approximately finite-dimensional von Neumann algebra equipped with a faithful normal semifinite trace then the space $\L^p(M,E)$ is $\K$-convex and 
$$
\K(\L^p(M,E))
\leq C.
$$
\end{lemma} 

\begin{proof}
Since the operator space $E$ is $\OK$-convex, the Banach space $S^p(E)$ is $\K$-convex. Hence the map $P_p \ot \Id_E \co \L^p(\Omega_0,E) \to \L^p(\Omega_0,E)$ is completely bounded. Using \cite[(3.1) and (3.6)]{Pis5}, this complete boundedness implies that $P_p \ot \Id_{\L^p(M,E)} \co \L^p(\Omega_0,\L^p(M,E)) \to \L^p(\Omega_0,\L^p(M,E))$ is a well-defined bounded operator with
$$ 
\norm{P_p \ot \Id_{\L^p(M,E)}}_{\L^p(\Omega_0,\L^p(M,E))\to \L^p(\Omega_0,\L^p(M,E))}
\leq \norm{P_p \ot \Id_E}_{\cb,\L^p(\Omega_0,E) \to \L^p(\Omega_0,E)}.
$$
Using \eqref{K-convex-Lp}, the end of the proof is obvious.
\end{proof}

\paragraph{Projections and $\K$-convexity} We will use the following result of Pisier \cite[Theorem 3.1]{Pis3}.

\begin{thm}
\label{Th-Pisier}
Let $X$ be a $\K$-convex Banach space. There exists a constant $C>0$ only depending
on the $\K$-convexity constant $\K(X)$ such that for any integer $n \geq 1$ and for any $n$-tuple $(P_1,\ldots,P_n)$ of mutually commuting contractive projections on $X$ we have
$$
\norm{\sum_{k=1}^n(\Id_X-P_k) \prod_{1\leq j\not=k\leq n} P_j}_{X \to X}
\leq C.
$$
\end{thm}
Moreover, we recall the following lemma \cite[Lemma 1.5]{Pis3} (see also \cite[Lemma 13.12]{DJT}). Note that \cite[Lemma 1.7]{Pis3} is used here.

\begin{lemma}
\label{Lemma-of-Pisier-on-projections} 
Let $n \geq 1$ be an integer. Let $X$ be a Banach space such that, for some $\lambda>1$, $X$ does not contain any subspace $\lambda$-isomorphic to $\ell_1^{n+1}$. Then there exist a real number $0<\rho<2$ (depending only on $\lambda$ and $n$) such that if $P_1,\ldots, P_n$ is any finite collection of mutually commuting norm one projections on $X$, then
\begin{equation*}
\Bgnorm{\prod_{1 \leq k \leq n} (\Id_{X}-P_k)}_{X\to X}
\leq \rho^n.
\end{equation*}
\end{lemma}

We say that $Y$ is representable in $X$ if for any $\epsi >0$, each finite- dimensional subspace of $Y$ is $(1+\epsi)$-isomorphic to a subspace of $X$ \cite[page 175]{DJT}.

\begin{prop}
\label{Finitely}
Suppose $1 \leq p<\infty$. Let $E$ be an operator space. Let $M$ be von Neumann algebra equipped with a normal semifinite faithful trace. Then $\L^p(M,E)$ is finitely representable in $S^p(E)$.
\end{prop}

\begin{proof}
We have $M = \ovl{\bigcup_\alpha M_\alpha}^{\w*}$ for some subalgebras $M_\alpha$ equipped with a finite trace $\tau_\alpha = \tau|_{M_\alpha}$. By \cite[Theorem 3.4]{Pis5}, we have $\L^p(M,E) = \ovl{\bigcup_\alpha \L^p(M_\alpha,\tau_\alpha,E)}$. Let $F$ be a finite-dimensional subspace of $\L^p(M,E)$. If $(x_1,\ldots,x_n)$ is a basis of $F$, there exist $\alpha$ and $y_1,\ldots,y_n \in \L^p(M_\alpha,\tau_\alpha,E)$ such that $\norm{x_k-y_k}_{\L^p(M,E)} < \epsi$ for any $k$. Then there exist a $(1+\epsi')$-isomorphism $\phi \co F \to \Ran(\phi) \subset \L^p(M_\alpha,\tau_\alpha,E)$. It suffices to $(1+\epsi')$-represent $\L^p(M_\alpha,E)$ in $S^p(E)$. This task is left the reader (adapt the method of the case 1 of the proof of \cite[Theorem 3.24]{ArK}).
\end{proof}


\begin{prop}
\label{Prop-contain-ell1}
Let $E$ be an operator space. Suppose $1<p<\infty$. Given an integer $n \geq 1$ and $\lambda > 1$, if $S^p(E)$ does not contain any subspace $\lambda$-isomorphic to $\ell^1_n$ then there is a number $\lambda'>1$ depending only on $p,n$ and $\lambda$ such that for any approximately finite-dimensional von Neumann algebra equipped with a faithful normal semifinite trace, the Banach space $\L^p(M,E)$ does not contain any subspace $\lambda'$-isomorphic to $\ell^1_n$.
\end{prop}

\begin{proof}
We let $\lambda'=\frac{\lambda+1}{2}>1$. Suppose that there exists a finite-dimensional subspace $F$ of $\L^p(M)$ which is $\lambda'$-isomorphic to $\ell^1_n$. Using Proposition \ref{Finitely}, taking $\epsi=\frac{\lambda-1}{\lambda+1}>0$, there exists a subspace $G$ of $S^p(E)$ which is $(1+\epsi)$-isomorphic to $F$. We deduce that $G$ is $\lambda'(1+\epsi)$-isomorphic to $\ell^1_n$, that is $\lambda$-isomorphic to $\ell^1_n$. Impossible.
\end{proof}

\section{Semigroups acting on $\L^p$-spaces associated to locally compact groups}
\label{sec:main}


Our first main result is the following theorem.

\begin{thm}
\label{Th-amenable-semigroup}
Suppose that $G$ is a (unimodular) second countable amenable locally compact group. Let $(T_t)_{t \geq 0}$ be a weak* continuous semigroup of selfadjoint contractive Fourier multipliers on the group von Neumann algebra $\VN(G)$. Suppose that $E$ is an $\OK$-convex operator space. Consider $1<p<\infty$. Then $(T_t)_{t \geq 0}$ induces a strongly continuous bounded analytic semigroup $(T_t \ot \Id_E)_{t \geq 0}$ of contractions on the vector-valued noncommutative $\L^p$-space $\L^p(\VN(G),E)$.
\end{thm}

\begin{proof} 
By \cite[Corollary 1.8]{DCH}, since $G$ is amenable, each Fourier multiplier $M_{\psi_t}\ov{\textrm{def}}=T_t$ is completely contractive on $\VN(G)$. By \eqref{Equa-transfer=}, we can consider the associated semigroup $\big(M_{\psi_t}^{\HS}\big)_{t \geq 0}$ of selfadjoint contractive Schur multipliers on the space $\B(\L^2(G))$. Since $G$ is amenable, the von Neumann algebra $\VN(G)$ is approximately finite-dimensional. Using the part 2 of Proposition \ref{prop-tensorisation of CP maps}, we see that the map $M_{\psi_t} \ot \Id_E$ extends to a complete contraction on the space $\L^p(\VN(G),E)$ for any $1 <p <\infty$. We deduce that
\begin{align*}
\bnorm{M_{\psi_t}^{\HS} \ot \Id_E}_{S^p_{G}(E) \to S^p_{G}(E)}
    &\leq \bnorm{M_{\psi_t}^{\HS} \ot \Id_E}_{\cb,S^p_{G}(E) \to S^p_{G}(E)}\\
		&\ov{\eqref{Equa-transfer=}}=\bnorm{M_{\psi_t} \ot \Id_E}_{\cb,\L^p(\VN(G),E) \to \L^p(\VN(G),E)} 
		\leq 1.
\end{align*}
Since the symbol of a completely bounded Fourier multiplier on $\VN(G)$ is equal almost everywhere to a continuous function, see e.g. \cite[Corollary 3.3]{Haa3}, we know that the symbol $\psi_t \co G \times G \to \R$ is continuous. For any finite family $\alpha=\{x_1,\ldots,x_{m_\alpha}\}$ of distinct elements of $G$, using Proposition \ref{Transference-Schur-discrete-continuous}, we can consider the contractive Schur multiplier $M_{\psi_{t,\alpha}}$ on $\B(\ell^2_{m_\alpha})$ defined by the matrix $\psi_{t,\alpha}=[\psi_{t}(x_ix_j^{-1})]_{1 \leq i,j \leq m_\alpha}$. As in the proof of \cite[Corollary 4.3]{Arh1}, there exists Schur multipliers $S_{1,t,\alpha}$ and $S_{2,t,\alpha}$ on $\B(\ell^2_{m_\alpha})$ such that
\begin{equation}
	\label{Wtalpha}
W_{t,\alpha}
\overset{\textrm{def}}=\begin{bmatrix}
    S_{1,t,\alpha} & M_{\psi_{t,\alpha}} \\
    (M_{\psi_{t,\alpha}})^\circ  & S_{2,t,\alpha} \\
\end{bmatrix}	
\end{equation}
is a completely positive unital selfadjoint Schur multiplier on $\B\big(\ell^2_{2m_\alpha}\big)$. For any $t \geq 0$, we see that
\begin{equation}
\label{bloc-prime}
\big(W_{\frac{t}{2},\alpha}\big)^2
\ov{\eqref{Wtalpha}}=\begin{bmatrix}
    S_{1,\frac{t}{2},\alpha}      & M_{\psi_{\frac{t}{2}},\alpha}\\
    \big(M_{\psi_{\frac{t}{2}},\alpha}\big)^\circ  & S_{2,\frac{t}{2},\alpha} \\
  \end{bmatrix}^2
=\begin{bmatrix}
    (S_{1,\frac{t}{2},\alpha})^2                        & (M_{\psi_{\frac{t}{2}},\alpha})^2 \\
    \big(M_{\psi_{\frac{t}{2}},\alpha}\big)^{\circ 2}  & (S_{2,\frac{t}{2},\alpha})^2 \\
		\end{bmatrix}
=\begin{bmatrix}
    (S_{1,\frac{t}{2},\alpha})^2                        & M_{\psi_{t},\alpha} \\
    \big(M_{\psi_{t},\alpha}\big)^{\circ }  & (S_{2,\frac{t}{2},\alpha})^2 \\
		\end{bmatrix}.
\end{equation}
Combining the construction of the noncommutative Markov chain of \cite[pages 4369-4370]{Ric} (see also \cite[Theorem 5.3]{HM}), for any $t \geq 0$, we infer that the Schur multiplier $(W_{\frac{t}{2},\alpha})^2$ admits a Rota dilation 
$$
\big((W_{\frac{t}{2},\alpha})^2\big)^{k}
=Q\E_k J,\qquad k \geq 1
$$
in the sense of \cite[Definition 10.2]{JMX} (extended to semifinite von Neumann algebras) where $J \co \B(\ell^2_{2m_\alpha}) \to N$ is a normal unital faithful $*$-representation into a von Neumann algebra $N$ (equipped with a trace) which preserve the traces, where $Q \co N \to \B(\ell^2_{2m_{\alpha}})$ is the trace preserving normal conditional expectation associated with $J$ and where the $\E_k$'s are conditional expectations onto von Neumann subalgebras of $N$. Recall that the von Neumann algebra $\Gamma_{-1}^{e}(\ell^{2,T})$ of \cite{Ric} is hyperfinite. Hence, the von Neumann algebra $N$ of the Rota Dilation is also hyperfinite. Note that we only need the case $k=1$ in the sequel of the proof. In particular, using the notation $\E\overset{\textrm{def}}=\E_1$ we have
\begin{equation}
	\label{QEJ}
\big(W_{\frac{t}{2},\alpha}\big)^2
=Q\E J 
\qquad \text{and} \qquad
\Id_{\B(\ell^2_{2m_{\alpha}})}=QJ.	
\end{equation}
We infer that
\begin{equation}
	\label{Eua-quelconque-1}
\Id_{\B(\ell^2_{2m_{\alpha}})}-\big(W_{\frac{t}{2},\alpha}\big)^2
\overset{\eqref{QEJ}}=QJ-Q\E J
=Q(\Id_N-\E)J.	
\end{equation}
Let $n$ be an integer. Note that we have
\begin{equation}
	\label{equa complexe}
	\Big(\Id_{\B(\ell^2_{2m_{\alpha}})}-\big(W_{\frac{t}{2},\alpha}\big)^2\Big)^{\ot n}
	\ov{\eqref{Eua-quelconque-1}}=\big(Q(\Id_N-\E)J\big)^{\ot n}
	=Q^{\ot n}(\Id_{N}-\E)^{\ot n} J^{\ot n}.
\end{equation}
For any integer $1 \leq k \leq n$, we consider the completely positive operator
$$
\Pi_k
=\Id_{\L^p(N)} \ot \cdots \ot \Id_{\L^p(N)} \ot \underbrace{\E}_{k} \ot
\Id_{\L^p(N)} \ot \cdots \ot \Id_{\L^p(N)}
$$
on the space $\L^p(N^{\ot n})$. By Proposition \ref{prop-tensorisation of CP maps}, we deduce that the $\Pi_k \ot \Id_E$'s induce a family of mutually commuting contractive projections on the Banach space $\L^p(N^{\ot n},E)$. Now, since $S^p(E)$ is $\K$-convex, using Proposition \ref{Prop-contain-ell1} and Lemma \ref{Lemma-of-Pisier-on-projections}, we can fix an integer $n \geq 1$ and $0<\rho<2$ such that
\begin{equation}
\label{equa strange}
\Bgnorm{\prod_{1 \leq k \leq n} \Big(\Id_{\L^p(N^{\ot n},E)}-(\Pi_k \ot \Id_E)\Big)}_{\L^p(N^{\ot n},E) \to \L^p(N^{\ot n},E)} 
\leq \rho^n.	
\end{equation}
Furthermore, we have
\begin{align}
\label{equazert}
\MoveEqLeft
\big(\Id_{\L^p(N)}-\E\big)^{\ot n}
=(\Id_{\L^p(N)}-\E) \ot \cdots \ot (\Id_{\L^p(N)}-\E)\\
&= \prod_{1 \leq k \leq n} \big(\Id_{\L^p(N)} \ot \cdots \ot \Id_{\L^p(N)} \ot \underbrace{(\Id_{\L^p(N)}-\E)}_{k} \ot \Id_{\L^p(N)} \ot \cdots \ot \Id_{\L^p(N)}\big)\nonumber \\
		&=\prod_{1 \leq k \leq n} \big(\Id_{\L^p(N^{\ot n})}-\Id_{\L^p(N)} \ot \cdots \ot \Id_{\L^p(N)} \ot \underbrace{\E}_{k} \ot \Id_{\L^p(N)} \ot \cdots \ot \Id_{\L^p(N)}\big)\nonumber\\
		&=\prod_{1 \leq k \leq n} \big(\Id_{\L^p(N^{\ot n})}-\Pi_k\big).\nonumber
\end{align}
Note that
\begin{align}
\MoveEqLeft
   \label{Calcul-bloc}         
\left(\Id_{S^p_2(S^p_{m_\alpha})}-\big(W_{\frac{t}{2},\alpha}\big)^2\right)^n		
\overset{\eqref{bloc-prime}}=\left(\begin{bmatrix}
 \Id_{S^p_{m_\alpha}}    & \Id_{S^p_{m_\alpha}}   \\
  \Id_{S^p_{m_\alpha}}    & \Id_{S^p_{m_\alpha}}   \\
\end{bmatrix}-\begin{bmatrix}
    \big(S_{1,\frac{t}{2},\alpha}\big)^2    & M_{\psi_{t},\alpha} \\
    \big(M_{\psi_{t,\alpha}}\big)^{\circ}  & \big(S_{2,\frac{t}{2},\alpha}\big)^2 \\
		\end{bmatrix}\right)^n	\\
		&=\left(\begin{bmatrix}
    \Id_{S^p_{m_\alpha}}-\big(S_{1,\frac{t}{2},\alpha}\big)^2 & \Id_{S^p_{m_\alpha}}-\big(M_{\psi_{t},\alpha}\big)^2 \\
    \Id_{S^p_{m_\alpha}}-\big(M_{\psi_t,\alpha}\big)^{\circ}  & \Id_{S^p_{m_\alpha}}-\big(S_{2,\frac{t}{2},\alpha}\big)^2 \\
		\end{bmatrix}\right)^n\nonumber\\
&=\begin{bmatrix}
\big(\Id_{S^p_{m_\alpha}}-(S_{1,\frac{t}{2},\alpha})^2\big)^n  & \big(\Id_{S^p_{m_\alpha}}-(M_{\psi_t,\alpha}\big)^n \\
\big(\Id_{S^p_{m_\alpha}}-(M_{\psi_t,\alpha})^{\circ}\big)^n  & \big(\Id_{S^p_{m_\alpha}}-(S_{2,\frac{t}{2},\alpha})^2\big)^n\\
		\end{bmatrix}. \nonumber
\end{align}   
Now, using Proposition \ref{prop-tensorisation of CP maps} in the third inequality, we obtain that
\begin{align*}
\MoveEqLeft 
\Bnorm{\big(\Id_{S^p_{m_\alpha}}-M_{\psi_t,\alpha}\big)^n \ot \Id_E}_{S^p_{m_\alpha}(E) \to S^p_{m_\alpha}(E)} \\
&\ov{\eqref{Calcul-bloc}}\leq \Bnorm{\Big(\Id_{S^p_2(S^p_{m_\alpha})}-\big(W_{\frac{t}{2},\alpha}\big)^2\Big)^n \ot \Id_{E}}_{S^p_2(S^p_{m_\alpha}(E)) \to S^p_2(S^p_{m_\alpha}(E))}\\
&\ov{\eqref{Majoration-norme-multiplicateur-1}}\leq \Bnorm{\Big(\Id_{S^p_2(S^p_{m_\alpha})}-\big(W_{\frac{t}{2},\alpha}\big)^2\Big)^{\ot n}\ot \Id_{E}}_{S^p(\B(\ell^2_{2 m_\alpha})^{\ot n},E) \to S^p(\B(\ell^2_{2{m_\alpha}})^{\ot n},E)}\\
&\ov{\eqref{equa complexe}}=\Bnorm{Q^{\ot n}(\Id_{\L^p(N)}-\E)^{\ot n} J^{\ot n} \ot \Id_E}_{S^p(\B(\ell^2_{2m_\alpha})^{\ot n},E) \to S^p(\B(\ell^2_{2m_\alpha})^{\ot n},E)}\\
&\leq \Bnorm{\big(\Id_{\L^p(N)}-\E\big)^{\ot n}\ot \Id_{E}}_{\L^p(N^{\ot n},E) \to \L^p(N^{\ot n},E)}\\
&\ov{\eqref{equazert}}=\Bgnorm{\prod_{1 \leq k \leq n} \Big(\Id_{\L^p(N^{\ot n},E)}-(\Pi_k\ot \Id_E)\Big)}_{\L^p(N^{\ot n},E) \to \L^p(N^{\ot n},E)} 
\ov{\eqref{equa strange}}\leq \rho^n.
\end{align*}
We deduce with Proposition \ref{Transference-Schur-discrete-continuous} that
\begin{equation}
	\label{rho-n}
\norm{\Big(\Id_{S^p_G}-M_{\psi_t}^{\HS}\Big)^n \ot \Id_E}_{S^p_G(E) \to S^p_G(E)}
	\leq \rho^n.	
\end{equation}
Hence, for any $t \geq 0$, using transference, we finally obtain
\begin{align*}
\MoveEqLeft
\norm{(\Id_{\L^p(\VN(G)}-T_t\big)^n \ot \Id_E}_{\L^p(\VN(G),E) \to \L^p(\VN(G),E)} \\
&=\norm{(\Id_{\L^p(\VN(G)}-M_{\psi_t}\big)^n \ot \Id_E}_{\L^p(\VN(G),E) \to \L^p(\VN(G),E)} \\
&=\norm{(M_{1-\psi_t}\big)^n \ot \Id_E}_{\L^p(\VN(G),E) \to \L^p(\VN(G),E)} 
=\norm{(M_{(1-\psi_t)^n} \ot \Id_E}_{\L^p(\VN(G),E) \to \L^p(\VN(G),E)}\\
&\overset{\eqref{Equa-transfer-leq}}\leq \bnorm{(M_{(1-\psi_t)^n}^{\HS} \ot \Id_E}_{S^p_{G}(E) \to S^p_{G}(E)}
\overset{\eqref{equavarphi}}= \bnorm{\big(\Id_{S^p_{G}}-M_{\psi_t}^{\HS}\big)^n \ot \Id_E}_{S^p_{G}(E) \to S^p_{G}(E)} 
		\overset{\eqref{rho-n}}\leq \rho^n.
\end{align*}
We conclude by Theorem \ref{Th-de-Beurling}.
\end{proof}

Now, we want use this result with an abelian locally compact group $G$. If $1 \leq p<\infty$, note that the definition of the Bochner space $\L^p(G,X)$ does not need an operator space structure on the Banach space $X$. So this theorem and the similar result \cite[Theorem 3.4]{Arh3} for amenable discrete groups led us to introduce the following definition.

\begin{defi}
\label{def-Kprime}
We say that a Banach space $X$ is $\K'$-convex if $X$ is isomorphic to a Banach space $E$ such that there exists an operator space structure on $E$ such that the vector-valued Schatten space $S^p(E)$ is $\K$-convex. 
\end{defi}

With the language of operator spaces, the condition on $E$ means that the operator space $E$ is $\OK$-convex. In the situation of this definition, $E$ is a closed subspace of $S^p(E)$, hence a $\K$-convex Banach space. Since $\K$-convexity is preserved under isomorphisms, we deduce that a $\K'$-convex Banach space is necessarily $\K$-convex. We conjecture the following. 

\begin{conj}
\label{Conj}
The class of $\K'$-convex Banach spaces coincides with the class of $\K$-convex Banach spaces.
\end{conj}

It might even be possible that each $\K$-convex Banach space admits an $\OK$-convex operator space structure. In Section \ref{K'-convex}, we will show that the class of $\K'$-convex Banach spaces shares the same nice properties of stability under usual operations than the class of $\K$-convex Banach spaces. This is an observation that goes in the direction of the conjecture.

Now, we state the consequence for abelian locally compact groups.

\begin{cor}
\label{Cor-abelien-semigroups}
Suppose that $G$ is a second countable abelian locally compact group. Let $(T_t)_{t \geq 0}$ be a weak* continuous semigroup of (not necessarily positive) selfadjoint contractive Fourier multipliers on $\L^\infty(G)$. Let $X$ be a $\K'$-convex Banach space. Consider $1<p<\infty$. Then $(T_t)_{t \geq 0}$ induces a strongly continuous bounded analytic semigroup $(T_t \ot \Id_X)_{t \geq 0}$ of contractions on the Bochner space $\L^p(G,X)$.
\end{cor}

\begin{proof}
Note that we have a $*$-isomorphism of von Neumann algebras between $\L^\infty(G)$ and $\VN(\hat{G})$ where $\hat{G}$ is the dual of $G$. Thus this corollary is a consequence of Theorem \ref{Th-amenable-semigroup}.
\end{proof}

By \cite[Corollary 0.1.1]{Lar} (and a duality argument), note that a semigroup satisfying the assumptions of Corollary \ref{Cor-abelien-semigroups} is precisely a semigroup $(T_t)_{t \geq 0}$ defined by $T_t(f)=\mu_t \ast f$ where $(\mu_t)_{t \geq 0}$ is a family of symmetric\footnote{\thefootnote. Recall that a bounded measure $\mu$ is symmetric if $\mu(-A)=\mu(A)$ for any measurable subset $A$ of $G$. It is clear that a bounded measure $\mu$ is symmetric if and only if its Fourier transform $\hat{\mu}$ is real valued.} bounded measures on $G$ with $\norm{\mu_t} \leq 1$ for any $t \geq 0$, forming a convolution semigroup, i.e. such that
$$
\mu_t \ast \mu_s
=\mu_{t+s}, \qquad t,s \geq 0.
$$
Thus, with a noncommutative extra assumption on $X$, this result gives an extension of the main result \cite[Theorem 1.2]{Pis3} of the paper of Pisier which is stated for a semigroup induced by a convolution semigroup $(\mu_t)_{t \geq 0}$ of symmetric \textit{probability measures} on $G$. In our result, we does not need that each $T_t$ is positive or that $T_t(1)=1$. The (unnecessary?) supplementary assumption is the price to pay to have the right to use noncommutative methods.


\section{Ritt operators acting on noncommutative $\L^p$-spaces associated to discrete compact groups}
\label{Sec-Ritt}

An operator $T \co X \to X$ is power-bounded if there exists a constant $C_0 \geq 0$ such that
$$
\norm{T^n}_{X \to X} 
\leq C_0, \qquad n \geq 0.
$$
Then a power-bounded $T$ is called a Ritt operator if there  exists a constant $C_1 \geq 0$ such that
\begin{equation}
\label{Definition-Ritt-operator}
\qquad n\norm{T^n -T^{n-1}} \leq C_1
\qquad n \geq 1.
\end{equation}
Ritt operators can be characterized by a spectral condition, as follows. Let 
$$
\D
=\{z \in \C\, :\, \vert z\vert<1\}
$$
be the open unit disc. For any bounded operator $T \co X \to X$, let $\sigma(T)$ denote the spectrum of $T$. Then  $T$ is a Ritt operator if and only if $\sigma(T) \subset \ovl{\D}$ and there exists a constant $K>0$ such that 
\begin{equation}
\label{Res}
\forall\, z \in \C \setminus \ovl{\D},\qquad
\norm{(z -T)^{-1}} \leq \,\frac{K}{\vert z-1\vert}.
\end{equation}
See \cite{Lyu}, \cite{ALM}, \cite{AFM}, \cite{Arh4}, \cite{Bl1} and \cite{Bl2} and references therein. Finally, recall the following characterization.


%
\begin{lemma}
\label{Lemma-carac-Ritt}
Let $T$ be a power-bounded linear operator on a Banach space $X$. Then $T$ is a Ritt operator if and only if the semigroup $\{e^{t(T-1)}\}_{t \geq 0}$ is bounded analytic and $\Sp(T) \subset \D \cup \{1\}$.
\end{lemma}

Let $I$ be an index set equipped with the discrete measure. Here, we will use the notation $e_{ij}^{\ot n} \overset{\textrm{def}}=e_{ij} \ot \cdots \ot e_{ij}$ for any $i,j \in I$. Recall the fundamental observation of \cite[Lemma 3.3]{Arh3}. We have for any positive integer $n$ and any matrix $b=[b_{ij}]$ which is finitely supported
\begin{equation}
\label{Absorption}
\norm{\sum_{i,j} e_{ij} \ot b_{ij}}_{S^p_I(E)}
=\norm{\sum_{i,j} e_{ij}^{\ot n} \ot b_{ij}}_{\L^p(\ovl{\ot}_{i=1}^n \B(\ell^2_I),E)}.	
\end{equation}
We will use the following lemma which is a noncommutative analogue of \cite[Lemma 13]{LLM} for Schur multipliers.

\begin{lemma}
\label{Lemma-consequence-entrelacement} 
Suppose $1 \leq p < \infty$. Let $E$ be an operator space. Let $n \geq 2$ and $m \geq 1$ be two integers. For any family $(M_{\phi_{rl}})_{1 \leq r \leq n, \leq l \leq m}$ of regular Schur multipliers on $S^p_I(E)$, we have
\begin{align}
\label{Majoration-norme-multiplicateur2}
\norm{\sum_{l=1}^{m} M_{\phi_nl} \cdots M_{\phi_{1l}} \ot \Id_E}_{S^p_I(E) \to S^p_I(E)}\leq \norm{\sum_{l=1}^{m} M_{\phi_{nl}} \ot \cdots \ot M_{\phi_{1l}} \ot \Id_E}_{S^p_I(E) \to S^p_I(E)}. 
\end{align} 
\end{lemma}

\begin{proof}
 For any positive integer $n$ and any function $a \co I \times I \to E$ finitely supported, using \eqref{Absorption}, we obtain
\begin{align*}
\MoveEqLeft  \Bgnorm{\bigg(\sum_{l=1}^{m} M_{\phi_nl} \cdots M_{\phi_{1l}} \ot \Id_E\bigg)\bigg(\sum_{i,j} e_{ij} \ot a_{ij}\bigg)}_{S^p_I(E)}
=\Bgnorm{\sum_{i,j} \sum_{l=1}^{m} \phi_{nl}(i,j) \cdots \phi_{1l}(i,j) e_{ij} \ot a_{ij}}_{S^p_I(E)}\\
&=\Bgnorm{\sum_{i,j} e_{ij} \ot \bigg( \sum_{l=1}^{m} \phi_{nl}(i,j) \cdots \phi_{1j}(i,j) a_{ij}\bigg)}_{S^p_I(E)}\\
&\ov{\eqref{Absorption}}=\Bgnorm{\sum_{i,j} e_{ij}^{\ot n} \ot \bigg( \sum_{l=1}^{m} \phi_{nl}(i,j) \cdots \phi_{1l}(i,j) a_{ij}\bigg)}_{\L^p(\ovl{\ot}_{i=1}^n \B(\ell^2_I),E)}\\
&=\Bgnorm{\sum_{i,j} \sum_{l=1}^{m} \phi_{nl}(i,j) \cdots \phi_{1l}(i,j)e_{ij}^{\ot n} \ot a_{ij}}_{\L^p(\ovl{\ot}_{i=1}^n \B(\ell^2_I),E)}\\
&=\Bgnorm{\bigg(\sum_{l=1}^{m} M_{\phi_nl} \ot \cdots \ot M_{\phi_{1l}} \ot \Id_E\bigg)\bigg(\sum_{i,j} e_{ij}^{\ot n} \ot a_{ij}\bigg)}_{\L^p(\ovl{\ot}_{i=1}^n \B(\ell^2_I),E)}\\
&\leq \norm{\sum_{l=1}^{m} M_{\phi_nl} \ot \cdots \ot M_{\phi_{1l}} \ot \Id_E}_{\L^p(\ovl{\ot}_{i=1}^n \B(\ell^2_I),E) \to \L^p(\ovl{\ot}_{i=1}^n \B(\ell^2_I),E)}
\Bgnorm{\sum_{i,j} e_{ij}^{\ot n} \ot a_{ij}}_{\L^p(\ovl{\ot}_{i=1}^n \B(\ell^2_I),E)}\\
&\ov{\eqref{Absorption}}=\norm{\sum_{l=1}^{m} M_{\phi_nl} \ot \cdots \ot M_{\phi_{1l}} \ot \Id_E}_{\L^p(\ovl{\ot}_{i=1}^n \B(\ell^2_I),E) \to \L^p(\ovl{\ot}_{i=1}^n \B(\ell^2_I),E)} \Bgnorm{\sum_{i,j} e_{ij} \ot a_{ij}}_{S^p_I(E)}.
\end{align*}
\end{proof}

The next theorem is a noncommutative variant of the main result of \cite{LLM}. Note also the result \cite[Theorem 5]{Xu3} for uniformly convex spaces and for squares of unital selfadjoint positive maps. 	
	%
This is our second main result. Note that we do not need any assumption of positivity in the following result.

\begin{thm}
\label{Main-result}
Let $G$ be a (unimodular) second countable amenable locally compact group. Let $E$ be a $\OK$-convex operator space. Let $M_\varphi \co \VN(G) \to \VN(G)$ be a Fourier multiplier which is the square of a contractive selfadjoint Fourier multiplier on $\VN(G)$. Suppose $1<p<\infty$. Then $M_{\varphi} \ot \Id_E \co \L^p(\VN(G),E) \to \L^p(\VN(G),E)$ is a Ritt operator.
\end{thm}

\begin{proof}
It is obvious that $M_\varphi \ot \Id_E$ is a well-defined contraction, hence a power-bounded operator. By assumption, there exists a selfadjoint contractive Fourier multiplier $M_\psi \co \VN(G) \to \VN(G)$ such that $M_\varphi$ is the square of $M_\psi$. Of course, we have $\psi^2=\varphi$. 

Since $G$ is amenable, the Fourier multiplier $M_\psi$ is completely contractive on $\VN(G)$. By \eqref{Equa-transfer=}, we can consider the associated selfadjoint contractive Schur multiplier $M_{\psi}^{\HS} \co \B(\L^2(G)) \to \B(\L^2(G))$ on the space $\B(\L^2(G))$. Using the part 2 of Proposition \ref{prop-tensorisation of CP maps}, we see that the map $M_{\psi} \ot \Id_E$ extends to a complete contraction on the space $\L^p(\VN(G),E)$. We deduce that
\begin{align*}
\bnorm{M_{\psi}^{\HS} \ot \Id_E}_{S^p_{G}(E) \to S^p_{G}(E)}
    &\leq \bnorm{M_{\psi}^{\HS} \ot \Id_E}_{\cb,S^p_{G}(E) \to S^p_{G}(E)}\\
		&\overset{\textrm{\eqref{Equa-transfer=}}}=\bnorm{M_{\psi} \ot \Id_E}_{\cb,\L^p(\VN(G),E) \to \L^p(\VN(G),E)} 
		\leq 1.
\end{align*}
Since the symbol of a completely bounded Fourier multiplier on $\VN(G)$ is equal almost everywhere to a continuous function, see e.g. \cite[Corollary 3.3]{Haa3}, we know that the symbol $\psi \co G \times G \to \R$ is continuous. For any finite family $\alpha=\{x_1,\ldots,x_{m_\alpha}\}$ of distinct elements of $G$, using Proposition \ref{Transference-Schur-discrete-continuous}, we can consider the contractive Schur multiplier $M_{\psi_{\alpha}}$ on $\B(\ell^2_{m_\alpha})$ defined by $\psi_{\alpha}=[\psi(x_ix_j^{-1})]$. As in the proof of \cite[Corollary 4.3]{Arh1}, there exists Schur multipliers $S_{1,\alpha}$ and $S_{2,\alpha}$ on $\B(\ell^2_{m_\alpha})$ such that
\begin{equation}
\label{W-alpha}
W_\alpha
\overset{\textrm{def}}=\begin{bmatrix}
    S_{1,\alpha} & M_{\psi_\alpha} \\
    (M_{\psi_\alpha})^\circ  & S_{2,\alpha} \\
\end{bmatrix}	
\end{equation}
is a completely positive unital self-adjoint Schur multiplier on $\B\big(\ell^2_{2m_\alpha}\big)$. We let $Z_\alpha \overset{\textrm{def}}=W^2_\alpha$. We see that
\begin{equation}
\label{bloc}
Z_\alpha
\ov{\eqref{W-alpha}}=\begin{bmatrix}
    S_{1,\alpha}      & M_{\psi_\alpha}\\
    \big(M_{\psi_\alpha}\big)^\circ  & S_{2,\alpha} \\
  \end{bmatrix}^2
=\begin{bmatrix}
    (S_{1,\alpha})^2                        & (M_{\psi_\alpha})^2 \\
    \big(M_{\psi_\alpha}\big)^{\circ 2}  & (S_{2,\alpha})^2 \\
		\end{bmatrix}.
\end{equation}
Combining the construction of the noncommutative Markov chain of \cite{Ric} (see also \cite[Theorem 5.3]{HM}) we infer that the Schur multiplier $Z_\alpha$ admits a Rota dilation 
$$
Z_\alpha^{k}
=Q\E_kJ,\qquad k \geq 1
$$
in the sense of \cite[Definition 10.2]{JMX} (extended to semifinite von Neumann algebras) where $J \co \B\big(\ell^2_{2m_\alpha}\big) \to N$ is a normal unital faithful $*$-representation into a von Neumann algebra (equipped with a trace) which preserve the traces, where $Q \co N \to \B\big(\ell^2_{2 m_\alpha}\big)$ is the conditional expectation associated with $J$ and where the $\E_k$'s are conditional expectations onto von Neumann subalgebras of $N$. Recall that the von Neumann algebra $\Gamma_{-1}^{e}(\ell^{2,T})$ of \cite{Ric} is hyperfinite. Hence, the von Neumann algebra $N$ of the Rota Dilation is also hyperfinite. In particular, if $\E\overset{\textrm{def}}=\E_1$ we have
\begin{equation}
\label{Dilation-esperance}
Z_\alpha
=Q\E J.	
\end{equation}
Let $n \geq 1$ be an integer. We can write
\begin{align}
\MoveEqLeft
\label{Eq-intermediaire}
 -n\big(Z_\alpha^n-Z_\alpha^{n-1}\big) 
=nZ_\alpha^{n-1}\big(\Id_{S^p_2(S^p_{m_\alpha})}-Z_\alpha\big)\\
&=\sum_{l=1}^n Z_\alpha^{n-1}\big(\Id_{S^p_2(S^p_{m_\alpha})}-Z_\alpha\big)
=\sum_{l=1}^n Z_\alpha^{n-l}\big(\Id_{S^p_2(S^p_{m_\alpha})}-Z_\alpha\big)Z_\alpha^{l-1}. \nonumber           
\end{align} 
Applying Lemma \ref{Lemma-consequence-entrelacement} with $m=n$ and the family $\{M_{\phi_{rl}}\}_{1 \leq r,l \leq n}$ of Schur multipliers on $S^p_2(S^p_{m_\alpha}(E))$ defined by
$$
M_{\phi_{ll}}
=\Id_{S^p_2(S^p_{m_\alpha})}-Z_\alpha
\qquad \text{and} \qquad
M_{\phi_{rl}}
=Z_\alpha \quad \text{if}\ r\not=l,
$$
we obtain
\begin{align*}
\MoveEqLeft
n\norm{\big(Z_\alpha^n-Z_\alpha^{n-1}\big) \ot \Id_{E}}_{S^p_2(S^p_{m_\alpha}(E)) \to S^p_2(S^p_{m_\alpha}(E)} \nonumber\\
&\overset{\eqref{Eq-intermediaire}}=n\norm{\sum_{l=1}^n Z_\alpha^{n-l}\big(\Id_{S^p_2(S^p_{m_\alpha})}-Z_\alpha\big)Z_\alpha^{l-1} \ot \Id_E}_{S^p_2(S^p_{m_\alpha}(E)) \to S^p_2(S^p_{m_\alpha}(E))} \nonumber \\
&\overset{\eqref{Majoration-norme-multiplicateur2}}\leq
\norm{\sum_{l=1}^n Z_\alpha^{\ot(n-l)} \ot \big(\Id_{S^p_2(S^p_{m_\alpha})}-Z_\alpha\big) \ot Z_\alpha^{\ot(l-1)} \ot \Id_E}_{S^p(\B(\ell^2_{2m_\alpha})^{\ot n},E) \to S^p(\B(\ell^2_{2m_\alpha})^{\ot n},E)}. \label{Tn}
\end{align*}
For any integer $1 \leq l \leq n$, we have
\begin{align*}
\MoveEqLeft
Z_\alpha^{\ot(n-l)} \ot (\Id_{S^p_2(S^p_{m_\alpha})}-Z_\alpha) \ot Z_\alpha^{\ot(l-1)}
\overset{\eqref{Dilation-esperance}}=(Q\E J)^{\ot(n-l)} \ot (\Id_{S^p_2(S^p_{m_\alpha})}-Q\E J) \ot (Q\E J)^{\ot(l-1)}\\
&=Q^{\ot n}\Big(\E^{\ot(n-l)} \ot \big(\Id_{\L^p(N)}-\E\big) \ot \E^{\ot(l-1)}\Big) J^{\ot n}.
\end{align*}
%
%
We deduce that
\begin{align}
\MoveEqLeft
n\norm{\big(Z_\alpha^n-Z_\alpha^{n-1}\big) \ot \Id_E}_{S^p_2(S^p_{m_\alpha}(E)) \to S^p_2(S^p_{m_\alpha}(E))}\nonumber\\
&\leq \norm{Q^{\ot n}\bigg(\sum_{l=1}^{n} \Big(\E^{\ot(n-l)} \ot \big(\Id_{\L^p(N)}-\E\big) \ot \E^{\ot(l-1)}\Big)\bigg) J^{\ot n}}_{S^p(\B(\ell^2_{2m_\alpha})^{\ot n},E) \to S^p(\B(\ell^2_{2m_\alpha})^{\ot n},E)} \nonumber\\
& \leq \norm{\sum_{l=1}^{n} \E^{\ot(n-l)} \ot (\Id_{\L^p(N)}-\E) \ot \E^{\ot(l-1)} \ot \Id_E}_{\L^p(\ovl{\ot}_{i=1}^{n} N,E) \to \L^p(\ovl{\ot}_{i=1}^{n} N,E)}.
\end{align} 
Now for any $l=1,\ldots,n$, define $P_l \co \L^p(\ovl{\ot}_{i=1}^{n} N,E) \to \L^p(\ovl{\ot}_{i=1}^{n} N,E)$ by
$$
P_l 
=\Id_{\L^p(N)}^{\ot(n-l)} \ot \E \ot \Id_{\L^p(N)}^{\ot(l-1)} \ot \Id_E.
$$ 
Then the operator in the right-hand side of the above inequality is equal to
\begin{equation}
\label{final}
\sum_{l=1}^n\Big(\Id_{\L^p(\ovl{\ot}_{i=1}^{n} N,E)} - P_l\big)
\prod_{1 \leq j \not=l \leq n} P_j.
\end{equation}
Since $\E$ is a conditional expectation, each $P_l$ is a contractive projection by Proposition \ref{prop-tensorisation of CP maps}. Moreover, the $P_l$'s mutually commute. By Lemma \ref{Lemma-K-convex}, the Banach space $\L^p(\ovl{\ot}_{i=1}^{n} N,E)$ is $\K$-convex and 
$$
\sup_{n \geq 1} \K\big(\L^p(\ovl{\ot}_{i=1}^{n} N,E)\big) 
<\infty.
$$ 
From Theorem \ref{Th-Pisier}, we deduce that the operators in (\ref{final}) are uniformly bounded. Consequently, the exists a positive constant $C$ such that
\begin{equation}
	\label{Equa-C}
\sup_{n \geq 1} n \norm{\big(Z_\alpha^n-Z_\alpha^{n-1}\big) \ot \Id_E}_{S^p_2(S^p_{m_\alpha}(E)) \to S^p_2(S^p_{m_\alpha}(E))} 
\leq C.	
\end{equation}
For any integer $n \geq 1$, note that
\begin{align}
\MoveEqLeft
	\label{bloc3}
Z_\alpha^n-Z_\alpha^{n-1}
\overset{\eqref{bloc}}=\begin{bmatrix}
    (S_{1,\alpha})^2  &   (M_{\psi_\alpha})^2         \\
    ((M_{\psi_\alpha})^2)^\circ  &   (S_{2,\alpha})^2 \\
  \end{bmatrix}^n
	-\begin{bmatrix}
    (S_{1,\alpha})^2  &  (M_{\psi_\alpha})^2         \\
    ((M_{\psi_\alpha})^2)^\circ  &  (S_{2,\alpha})^2 \\
  \end{bmatrix}^{n-1}\\
&=\begin{bmatrix}
    (S_{1,\alpha})^{2n}-(S_{1,\alpha})^{2(n-1)}  & (M_{\psi_\alpha})^{2n}-(M_{\psi_\alpha})^{2(n-1)}     \\
((M_{\psi_\alpha})^{2n})^\circ-((M_{\psi_\alpha})^{2(n-1)})^\circ  &   (S_{2,\alpha})^{2n}-(S_{2,\alpha})^{2(n-1)} \\
  \end{bmatrix}. \nonumber           
\end{align}
For any integer $n \geq 1$, we obtain that
\begin{align*}
\MoveEqLeft 
 n\Bnorm{(M_{\psi_\alpha})^{2n}-(M_{\psi_\alpha})^{2(n-1)} \ot \Id_E}_{S^p_{m_\alpha}(E) \to S^p_{m_\alpha}(E)}\\
&\overset{\eqref{bloc3}}\leq  n \norm{\big(Z_\alpha^n-Z_\alpha^{n-1}\big) \ot \Id_E}_{S^p_2(S^p_{m_\alpha}(E)) \to S^p_2(S^p_{m_\alpha}(E))} 
\overset{\eqref{Equa-C}}\leq C.
\end{align*}
We deduce with Proposition \ref{Transference-Schur-discrete-continuous} that
\begin{align*}
\MoveEqLeft
n\Bnorm{(M_{\psi})^{\HS})^{2n}-(M_{\psi}^{\HS})^{2(n-1)} \ot \Id_E}_{S^p_{G}(E) \to S^p_{G}(E)}
\leq C.
\end{align*}
Using \eqref{equavarphi}, since $\psi^2=\varphi$, this means that 
\begin{equation}
\label{}
n\Bnorm{\Big(\big(M_{\varphi}^{\HS}\big)^n-\big(M_{\varphi}^{\HS}\big)^{n-1}\Big) \ot \Id_E}_{S^p_{G}(E) \to S^p_{G}(E)} 
\leq C.	
\end{equation}
Using transference, we finally obtain that
\begin{align*}
\MoveEqLeft
\sup_{n \geq 1} n\Bnorm{\Big(\big(M_{\varphi}\big)^n-\big(M_{\varphi}\big)^{n-1}\Big) \ot \Id_E}_{\L^p(\VN(G),E) \to \L^p(\VN(G),E)} \\
&=\sup_{n \geq 1} n\Bnorm{\big(M_{\varphi^n}-M_{\varphi^{n-1}}\big) \ot \Id_E}_{\L^p(\VN(G),E) \to \L^p(\VN(G),E)}\\
&=\sup_{n \geq 1} n\bnorm{M_{\varphi^n-\varphi^{n-1}} \ot \Id_E}_{\L^p(\VN(G),E) \to \L^p(\VN(G),E)}\\
&\overset{\eqref{Equa-transfer-leq}}\leq \sup_{n \geq 1} n\Bnorm{M_{\varphi^n-\varphi^{n-1}}^\HS \ot \Id_E}_{S^p_{G}(E) \to S^p_{G}(E)}\\
&\overset{\eqref{equavarphi}}{=} \sup_{n \geq 1} n\Bnorm{\Big(\big(M_{\varphi}^{\HS}\big)^n-M_{\varphi}^{\HS}\big)^{n-1}\Big) \ot \Id_E}_{S^p_{G}(E) \to S^p_{G}(E)} \leq C <\infty.
\end{align*}
By \eqref{Definition-Ritt-operator}, we conclude that $M_\varphi \ot \Id_E \co \L^p(\VN(G),E) \to \L^p(\VN(G),E)$ is a Ritt operator.
\end{proof}

Similarly to Corollary \ref{Cor-abelien-semigroups}, we have the following result for abelian locally compact groups.

\begin{cor}
\label{Cor-groupes-commutatifs}
Let $G$ be a second countable abelian locally compact group. Let $X$ be a $\K'$-convex Banach space. Let $T \co \L^\infty(G) \to \L^\infty(G)$ be a Fourier multiplier which is the square of a contractive selfadjoint Fourier multiplier on $\L^\infty(G)$. Suppose $1<p<\infty$. Then $T \ot \Id_E \co \L^p(G,E) \to \L^p(G,E)$ is a Ritt operator.
\end{cor}

By \cite[Corollary 0.1.1]{Lar} (and a duality argument), note that an operator satisfying the assumptions of Corollary \ref{Cor-groupes-commutatifs} is precisely an operator $T$ defined by $T(f)=\mu \ast \mu \ast f$ where $\mu$ is a symmetric\footnote{\thefootnote. Recall that a bounded measure $\mu$ is symmetric if $\mu(-A)=\mu(A)$ for any measurable subset $A$ of $G$. It is clear that a bounded measure $\mu$ is symmetric if and only if its Fourier transform $\hat{\mu}$ is real-valued.} bounded measure on $G$ with $\norm{\mu} \leq 1$. Thus, with a noncommutative extra assumption on $X$, this result gives an extension of the main result \cite[Theorem 11]{LLM} of the paper of Lancien and Le Merdy which is stated for a convolution operator induced by the square of a symmetric \textit{probability measure} on $G$. In our result, we does not need that each $\mu$ is positive or that $T(1)=1$. 

Similarly, we can prove the next results. We skip the details. 

\begin{thm}
\label{Th-main-Schur}
Let $E$ be a $\OK$-convex operator space and $I$ be an index set equipped with the discrete measure. Let $M_\varphi \co \B(\ell^2_I) \to \B(\ell^2_I)$ be a Schur multiplier which is the square of a contractive selfadjoint Schur multiplier on $\B(\ell^2_I)$. Suppose $1<p<\infty$. Then $M_{\varphi} \ot \Id_E \co S^p_I(E) \to S^p_I(E)$ is a Ritt operator on the vector-valued Schatten space $S^p_I(E)$.
\end{thm}

If the group is discrete, we can remove the assumption of second countability since transference \cite{NR} remains valid.

\begin{thm}
\label{Th-discrete-groups}
Let $G$ be a discrete group. Let $E$ be a $\OK$-convex operator space. Let $M_\varphi \co \VN(G) \to \VN(G)$ be a Fourier multiplier which is the square of a contractive selfadjoint Fourier multiplier on $\VN(G)$. Suppose $1<p<\infty$. Then $M_{\varphi} \ot \Id_E \co \L^p(\VN(G),E) \to \L^p(\VN(G),E)$ is a Ritt operator.
\end{thm}

\begin{remark} \normalfont
We does not know if we can remove the ``square'' from the statements.
\end{remark}

\section{ $\K'$-convex Banach spaces}
\label{K'-convex}

In this section, we examine the impact of usual operations on the class of $\K'$-convex Banach spaces defined in Definition \ref{def-Kprime}. The following result shows that this class has very good stability properties. Note these properties are also shared by the class of $\K$-convex Banach spaces. Furthermore, this result allows use to give a lot of examples of $\K'$-Banach spaces.

\begin{prop}
Let $X$ be a $\K'$-convex Banach space.
\begin{enumerate}
\item The dual Banach space $X^*$ is $\K'$-convex.

\item If $\Omega$ is a (localizable) measure space and if $1<p<\infty$, the Bochner space $\L^p(\Omega,X)$ is $\K'$-convex.

\item If $Y$ is a closed subspace of $X$, then both $Y$ and $X/Y$ are $\K'$-convex.

\item If $Y$ is a Banach space and if $0<\theta<1$ then the interpolation space $(X,Y)_\theta$ is $\K'$-convex.
	
\item If $Y$ is a $\K'$-convex Banach space and if $1<p<\infty$ then the Banach space $X \oplus_p Y$ is $\K'$-convex.
	
		\item Let $I$ be an index set and let $\ul$ be an ultrafilter on $I$. Then the ultrapower $X^\ul$ is $\K'$-convex.
	
\item If $Y$ is a Banach space isomorphic to $X$ then $Y$ is $\K'$-convex.
	
\item Suppose $1<p<\infty$. A noncommutative $\L^p$-space $\L^p(M)$ is $\K'$-convex.
\end{enumerate}

\end{prop} 

\begin{proof}
We consider an isomorphism $T \co X \to E$ where $E$ is a Banach space equipped with an operator space structure such that $S^p(E)$ is $\K$-convex for any $1<p<\infty$.

1) Note that $X^*$ is isomorphic to $E^*$. We equip the dual Banach space $E^*$ with the canonical operator space structure \cite [page 41]{ER} induced by the one of $E$. By \cite[Corollary 1.8]{Pis5}, we have $S^p(E^*)=(S^{p^*}(E))^*$ isometrically. By \cite[Corollary 13.7 and Theorem 13.15]{DJT}, we infer that the Banach space $S^p(E^*)$ is $\K$-convex. So we conclude that the Banach space $X^*$ is $\K'$-convex.

2) Note that the Banach space $\L^p(\Omega,X)$ isomorphic to $\L^p(\Omega,E)$. We equip the Banach space $\L^p(\Omega,E)$ with its natural operator space structure defined in \cite[page 32]{Pis5}. By \cite[(3.6)]{Pis5}, we have $S^p(\L^p(\Omega,E))=\L^p(\Omega,S^p(E))$ isometrically. By \cite[Lemma 10]{LLM}, we infer that $S^p(\L^p(\Omega,E))$ is $\K$-convex. We deduce that the Banach space $\L^p(\Omega,X)$ is $\K'$-convex. 

3) Note that the Banach space $Y$ is isomorphic to $T(Y)$ and that $X/Y$ is isomorphic to $E/T(Y)$. We equip the Banach space $T(Y)$ with the operator space structure induced by the isometric inclusion $T(Y) \subset E$ and the quotient $E/T(Y)$ with its canonical operator space structure described in \cite[pages 37-38]{ER}. By \cite[Theorem 1.1]{Pis5} and by the injectivity of the Haagerup tensor product \cite[Proposition 9.2.5]{ER}, we have an isometric inclusion $S^p(T(Y)) \subset S^p(E)$. Using \cite[Theorem 13.15]{DJT} and \cite[Theorem 9]{Gie}, we infer that $S^p(T(Y))$ is $\K$-convex. We deduce that the Banach space $Y$ is $\K'$-convex. By \cite[(3.4)]{Pis5}, we have an isometry $S^p(E/T(Y))=S^p(E)/S^p(T(Y))$. By \cite[Theorem 13.3 and Proposition 11.11]{DJT}, we obtain that $S^p(E/T(Y))$ is $\K$-convex. So we conclude that the Banach space $X/Y$ is $\K'$-convex.

4) Note that the Banach space $(X,Y)_\theta$ is isomorphic to the underling Banach space $(E,Y)_\theta$ of the operator space $(E,\min(Y))_\theta$. By \cite[(1.5)]{Pis5}, we have 
$$
S^p((E,\min(Y))_\theta)
=(S^p(E),S^p(\min(Y)))_\theta
$$ 
isometrically. By \cite[Th\'eor\`eme 1]{Pis8}, we deduce that the Banach space $S^p((E,\min(Y))_\theta)$ is $\K$-convex. Hence the Banach space $(X,Y)_\theta$ is $\K'$-convex.
	
5) Note that $Y$ is isomorphic to a Banach space $F$ such that there exists an operator space structure on $F$ such that $S^p(F)$ is $\K$-convex. Observe that $X \oplus_p Y$ is isomorphic to $E \oplus_p F$. We equip the Banach space $E \oplus_p F$ with the operator space structure described in \cite[pages 34-35]{Pis5}. Using \cite[(2.9)]{Pis5}, we see that 
$$
S^p(E \oplus_p F)=S^p(E) \oplus_p S^p(F).
$$ 
By \cite[Theorem 13.15]{DJT} and \cite[Lemma 11]{Gie}, we infer that $S^p(E \oplus_p F)$ is $\K$-convex. So we conclude that the Banach space $X \oplus_p Y$ is $\K'$-convex.	
	
6) Note that the Banach spaces $X^\ul$ and $E^\ul$ are isomorphic. We equip $E^\ul$ with its canonical operator space structure described in \cite[page 185]{ER}. Recall that $\K$-convexity is a super-property by \cite[page 1304]{Mau} and \cite[Theorem 13.15]{DJT}. Using \cite[page 1304]{Mau}, we see that the Banach space $S^p(E)^\ul$ is $\K$-convex. By \cite[Lemma 5.4]{Pis5}, for any integer $n \geq 1$, we have an isometry $S^p_n(E^\ul)=(S^p_n(E))^\ul$. Moreover, we have an isometric inclusion $(S^p_n(E))^\ul \subset (S^p(E))^\ul$. We deduce that each $S^p_n(E^\ul)$ is $\K$-convex and that
$$
\K\big(S^p_n\big(E^\ul\big)\big)
=\K\big(S^p_n(E)^\ul\big)
\leq \K\big(S^p\big(E\big)^\ul\big).
$$
Now $\cup_{n \geq 1} S^p_n(E^\ul)$ is dense in $S^p(E^\ul)$ by \cite[Remark page 23]{Pis5}. If $P$ is the map defined in \eqref{Projection-Rademacher}, we infer that the map $P \ot \Id_{S^p(E^\ul)} \co \L^2(\Omega_0,S^p(E^\ul)) \to \L^2(\Omega_0,S^p(E^\ul))$ is a well-defined bounded operator. Hence the Banach space $S^p(E^\ul)$ is $\K$-convex. We conclude that $X^\ul$ is $\K'$-convex.

7) It is obvious.

8) We equip $\L^p(M)$ with its natural operator space structure. By \cite[(3.6)']{Pis5}, we have $S^p(\L^p(M))=\L^p(M,S^p)$ isometrically. By \cite[Corollary 5.5]{PiX}, the noncommutative $\L^p$-space $\L^p(M,S^p)$, has type $>1$. By \cite[Theorem 13.3]{DJT}, we deduce that $\L^p(M,S^p)$ is $\K$-convex. We conclude that $\L^p(M)$ is $\K'$-convex.
\end{proof}

\textbf{Acknowledgements} 
The author will thank Christian Le Merdy and Christoph Kriegler for a discussion on Proposition \ref{Finitely}.


\vspace{0.2cm}
\footnotesize{
\noindent C\'edric Arhancet\\ 
\noindent13 rue Didier Daurat, 81000 Albi, France\\
URL: \href{http://sites.google.com/site/cedricarhancet}{http://sites.google.com/site/cedricarhancet}\\
cedric.arhancet@protonmail.com\\
}

\end{document}